\documentclass[a4paper,12pt]{article}

\usepackage[utf8]{inputenc}
\usepackage{amsmath, amsfonts, amssymb, amsthm}
\usepackage{mathrsfs}
\usepackage{fullpage}

\usepackage{mathabx}%for special symbols

\usepackage[english]{babel}

\usepackage{indentfirst}

\usepackage{bbm}

\numberwithin{equation}{section}

\newtheorem{theorem}{Theorem}

\newtheorem{lemma}{Lemma}
\newtheorem{prop}{Proposition}
\newtheorem{corollary}{Corollary}

\newcommand{\Z}{\mathbb{Z}}
\newcommand{\N}{\mathbb{N}}
\newcommand{\FF}{\mathcal{F}}
\newcommand{\supp}{\mathrm{supp}}

\title{On the Littlewood--Paley--Rubio de Francia inequality for the bounded Vilenkin systems}

\author{Anton Tselishchev\thanks{The research is supported  by the Foundation for the Advancement of
		Theoretical Physics and Mathematics ``BASIS'' and by a grant from the Government of Russian Federation, agreement 075-15-2019-1620}}

\date{}

\begin{document}
	\maketitle
	
	\begin{abstract}
		The one-sided Littlewood--Paley inequality for arbitrary intervals was proved by Rubio de Francia. Later, N. Osipov proved its analogue for the system of Walsh functions. In this paper, this inequality is proved for more general Vilenkin systems.
	\end{abstract}

	\section{Introduction}

Suppose that $\{I_j\}_{j\in\Z}$ is  a set of pairwise disjoint intervals in $\Z$ and $f$ is a function on $\mathbb{T}$. Let us denote by $P_j$ the operator defined by the relation $(P_j f)\hat{ }=\chi_{I_j}\hat{f}$ where $\hat{f}$ is the Fourier transform of the function $f$ (in our case it means that it is a sequence of its Fourier coefficients). In the paper \cite{RdF} Rubio de Francia proved that for any $p\ge 2$ the following inequality holds:
$$
\Big\|\Big(\sum_j |P_j f|^2\Big)^{1/2}\Big\|_p\lesssim \|f\|_p.
$$
We write $A\lesssim B$ to indicate that the left hand side of the inequality does not exceed the right rand side multiplied by some positive constant. Here this constant does not depend on the family $\{I_j\}$ and the function $f$.

By duality it is easy to see that this inequality is equivalent to the following:
\begin{equation}
	\Big\| \sum_j f_j \Big\|_p\lesssim \Big\|\Big(\sum_j |f_j|^2\Big)^{1/2}\Big\|_p, \qquad 1<p\le 2,
	\label{RdFclass}
\end{equation}
where the functions $f_j$ are such that $\supp \widehat{f}_j\subset I_j$.

We will prove such inequality for Vilenkin systems instead of the system of exponents. First of all, we need to introduce some notation.

We denote by $L^p$ the space $L^p[0,1]$. Also, let $L^p(\ell^2) = L^p([0,1]; \ell^2)$ (this space consists of $\ell^2$-valued functions defined on $[0,1]$). Usually, there will be no difference between scalar-valued and $\ell^2$-valued functions; if a function $f$ is $\ell^2$-valued, then $|f|$ should be understood as $\|f(\cdot)\|_{\ell^2}$.

Suppose that $\{p_i\}_{i=1}^\infty$ is a sequence of integers not less than $2$. In this case the Vilenkin system corresponding to this sequence is the set of characters (that is, continuous homomorphisms into the unit circle $\{z\in\mathbb{C}: |z|=1\}$) on the group
$$
G=\prod_{i=1}^\infty \Z_{p_i}.
$$
Such systems of functions were introduced in N. Ya. Vilenkin's paper  \cite{Vil}.

We denote the product $p_1 p_2\ldots p_l$ by $m_l$ (and we set $m_0=1$). It will be convenient for us to assume that these Vilenkin functions are defined on the interval $[0,1]$ (sometimes they are assumed to be defined on $[0,1)$, see for example the definitions in the books \cite{GolEfSkvor,Weisz}; however, it is not important in the context of our paper). Indeed, the set $G$ (with the Haar measure $G$) can be identified with the interval $[0,1]$ (up to the countable set of points) by the map
$$
G\ni (a_1, a_2, \ldots)\mapsto \sum_{i=1}^\infty \frac{a_i}{m_i}.
$$
Here we suppose that $0\le a_i\le p_i-1$. It is not difficult to see that such map is measure-preserving.

Now we divide the segment $[0,1]$ into $p_1$ equal segments and denote by $r_1$ the functions which is equal to $e^{2\pi i k/p_1}$ on the $k$th segment (we start the enumeration of segments from $0$). Next, we do the same with each of our $r_1$ segments: divide it into $p_2$ parts and denote by $r_2$ the function which is equal to $e^{2\pi i k/p_2}$ on the $k$th part. We resume this procedure and get the functions $r_i$ which are the analogues of the classical Rademacher functions.

All Vilenkin functions are the products of these generalized Rademacher functions. Specifically, each number $n\in\Z_+$ (this symbol denotes the set of non-negative integers) can be written in ``$(m)$-radix numeral system'', that is, it can be represented as $n=\alpha_1+\alpha_2 m_1+\ldots \alpha_k m_{k-1}$, where $0\le \alpha_i\le p_i-1$. It will be convenient for us to use the following notation for such representation: 
$$
n\sim \begin{pmatrix}
	m_{k-1} & \ldots & m_1 & m_0\\
	\alpha_k & \ldots & \alpha_2 & \alpha_1
\end{pmatrix}.
$$
In this case, the Vilenkin function $w_n$ is uqual to $r_1^{\alpha_1} r_2^{\alpha_2}\ldots r_k^{\alpha_k}$. This representation and enumeration of Vilenkin functions may be found for instance in the paper \cite{Wat}. The Vilenkin system is an orthonormal basis in $L^2$. We also note that if all numbers $p_i$ are equal to $2$, the we get the classical Rademacher and Walsh functions as the special case of the above construction.

Now we pass to the formulation of the main result of this paper. We will assume that the sequence $p_i$ is bounded: $p_i\le M$ for some $M>2$. Such Vilenkin systems are called bounded. This is an important assumption for our argument, which will be clear from what follows. The symbol $\widehat{f}$ will denote the Vilenkin--Fourier coefficients of the function $f$: $\widehat{f}(n)=(f, w_n)=\int f\bar{w}_n$. It is clear that in this case $f=\sum_{n\in\N_0} \widehat{f}(n) w_n$ (if $f\in L^2$). 

For the remaining part of the paper we fix some bounded sequence $\{p_i\}_{i\in \N}$ and the corresponding Vilenkin system $\{w_n\}_{n\in\Z_+}$. We will prove the following theorem.

\begin{theorem}
	Suppose that $\{I_s\}$ is a family of pairwise disjoint segments in $\Z_+$ and functions $f_s$ are such that $\supp \widehat{f}_s\subset I_s$ \emph{(}that is, each of these functions is a ``Vilenkin polynomial''\emph{)}. Then for any $1<p\le 2$ the following inequality holds:
	$$
	\Big\| \sum_s f_s \Big\|_p\lesssim \Big\|\Big(\sum_s |f_s|^2\Big)^{1/2}\Big\|_p.
	$$
\end{theorem}

We note that the similar statement for Walsh functions was proved by N. Osipov in the paper \cite{Osip}. However, it turned out that the generalization for the Vilenkin systems is not straightforward --- we will use the special (not martingale) square function and also we will need a refinement of the combinatorial constructions from the paper \cite{Osip}. 

Besides that, the classical Rubio de Francia's inequality (\ref{RdFclass}) was later proved for $p=1$ by Bourgain in his paper \cite{Bour} and also for all $p\in(0,2]$ by Kislyakov and Parilov in \cite{KisLP}. In the context of Vilenkin functions the case $p\le 1$ will be discussed briefly in Section 4 of this paper.

The author is kindly grateful to his scientific advisor  S. V. Kislyakov for posing this problem and to N. N. Osipov and V. A. Borovitsky for the fruitful discussions.

\section{Auxiliary statements}

Let $k, l \in \Z_+$ be integers that are represented in $(m)$-radix numerical system in the following way:
$$
k=\alpha_1 + \alpha_2 m_1+\ldots+\alpha_j m_{j-1}, \quad l=\beta_1 +\beta_2 m_1 +\ldots+ \beta_j m_{j-1}.
$$
It is easy to see that the product of functions  $w_k$ and $w_l$ is the function $w_{k\dotplus l}$, where $k\dotplus l$ is the number of the following form:
$$
\begin{pmatrix}
	m_{j-1} & \ldots & m_1 & m_0\\
	(\alpha_j+\beta_j)\, \mathrm{mod}\, p_j & \ldots & (\alpha_2+\beta_2)\, \mathrm{mod}\, p_2 & (\alpha_1+\beta_1)\, \mathrm{mod}\, p_1
\end{pmatrix}.
$$
We denote by $(\dotdiv k)$ the inverse of $k$ with respect to the operation $\dotplus$. 

The symbol $\FF_k$ denotes the $\sigma$-algebra generated by the intervals  $[jm_k^{-1},(j+1)m_k^{-1})$, $0\le j\le m_k-1$, on $[0,1)$. The operator $\mathbb{E}_k$ defined by the relation 
$$
\mathbb{E}_k f =\sum_{n=0}^{m_k-1} (f,w_n)w_n,
$$
is the  conditional expectation with respect to $\FF_k$. The corresponding martingale differences have the following form:
$$
\Delta_k f=\mathbb{E}_k f-\mathbb{E}_{k-1} f=\sum_{n=m_{k-1}}^{m_k-1} (f, w_n)w_n.
$$
We also use the notation $\Delta_0 f$ for the function $(f, w_0)w_0$ (this is simply the function equal to $\int f$ on $[0,1]$). We note that the boundedness of the sequence $\{p_i\}$ is equivalent to the regularity of the filtration $\{\FF_k\}$ (by regularity we mean that for any set $e\in\FF_k$ there is a set $e'\in\FF_{k-1}$ which contains $e$ such that its measure does not exceed $M$ times the measure of $e$).

The martingale square function $Sf$ is defined as
$$
Sf=\sqrt{\sum_{j=0}^\infty |\Delta_j f|^2}.
$$
It is well known that for $p>1$ we have $\|Sf\|_p\asymp \|f\|_p$ (see for example \cite[\S 2.2]{Weisz}).

However, when we work with Vilenkin systems it is often more convenient to use another square function. We define the operators $\Delta_{k,l}$ in a following way:
$$
\Delta_{k,l}f=\sum_{n=lm_{k-1}}^{(l+1)m_{k-1}-1} (f,w_n)w_n, \quad 1\le l\le p_k-1.
$$
The square function we are going to use is defined as
$$
\widetilde{S} f = \Big(  |\Delta_0 f|^2 + \sum_{k=1}^\infty \sum_{l=1}^{p_k-1} |\Delta_{k,l} f|^2 \Big)^{1/2}.
$$
Such operator was already introduced in the paper \cite{Wat}. It was also proved that its $L^p$-norm can be estimated by the $L^p$-norm of the function $f$. We will present the short proof of this fact here for the sake of completeness (also, this proof works equally well for $\ell^2$-valued functions $f$).

\begin{lemma}
	For $1<p<\infty$ and $f\in L^p$ we have $\|\widetilde{S} f\|_p \asymp \|f\|_p$.
\end{lemma}
\begin{proof}
	First we note that we have the pointwise estimate $Sf \lesssim \widetilde{S}f$ and therefore we only need to prove the inequality $\|\widetilde{S} f\|_p \lesssim \|Sf\|_p$. For any number $k$ we fix $l_k$, $1\le l_k\le p_k-1$. Our goal is to prove the estimate
	$$
	\Big\| \Big( \sum_k |\Delta_{k, l_k} f|^2 \Big)^{1/2} \Big\|_p \lesssim \|Sf\|_p.
	$$
	The desired inequality follows from this estimate since $\widetilde{S}f$ can be represented as the square root of the sum of the squares of certain number (not more than $M$) quantities from the left hand side of the estimate.
	
	Consider the family of functions $(\Delta_1 f, \Delta_2 f, \ldots)=(f_1, f_2, \ldots)$. We note that the following relation holds:
	$$
	\Delta_{k, l_k}f=w_{l_k m_{k-1}}\mathbb{E}_{k-1}[w_{l_k m_{k-1}}^{-1} f_k].
	$$
	
	This relation follows from the fact that $w_n^{-1}w_m=w_{m\dotdiv n}$ and
	$$
	[l_km_{k-1}, (l_k+1)m_{k-1}-1]\dotdiv l_km_{k-1}=[0, m_{k-1}-1].
	$$
	Therefore, we have:
	\begin{multline*}
		\Big\| \Big( \sum_k |\Delta_{k, l_k} f|^2 \Big)^{1/2} \Big\|_p = \Big\| \Big( \sum_k \big|\mathbb{E}_{k-1}[w_{l_k m_{k-1}}^{-1} f_k]\big|^2 \Big)^{1/2} \Big\|_p \\ \lesssim \Big\| \Big( \sum_k \big|w_{l_k m_{k-1}}^{-1}f_k\big|^2 \Big)^{1/2} \Big\|_p=\|Sf\|_p.
	\end{multline*}
	Here we used that for any positive integers $n_k$ the following inequality holds:
	$$
	\|\{\mathbb{E}_{n_k} g_k\}\|_{L^p(\ell^2)}\lesssim \|\{g_k\}\|_{L^p(\ell^2)}.
	$$
	This inequality follows from the easy observation that the (martingale) square function of the $\ell^2$-valued function $\{\mathbb{E}_{n_k} g_k\}$ is less than or equal to the square function of $\{g_k\}$. Indeed, we have:
	$$
	S(\{g_k\})=\Big(\sum_{k}\sum_{j=0}^\infty |\Delta_j g_k|^2\Big)^{1/2}\ge \Big(\sum_{k}\sum_{j=0}^{n_k} |\Delta_j g_k|^2\Big)^{1/2}=S(\{\mathbb{E}_{n_k} g_k\}).
	$$
\end{proof}

As we already mentioned, the above lemma holds for both scalar-valued and $\ell^2$-valued functions $f$. We also note that we used the boundedness of the Vilenkin system here --- this lemma fails for unbounded systems, this issue is also discussed in the paper \cite{Wat}.

We will also need another simple property of the operators $\Delta_{k,l}$. Note that if the support of the function $f$ is contained in the set $e_k\in\FF_{k-1}$, then the support of the function $\Delta_k f$ is also contained in $e_k$ (since $\Delta_k f=\mathbb{E}_k f - \mathbb{E}_{k-1}f$, and the operators $\mathbb{E}_k$ and $\mathbb{E}_{k-1}$ simply average the function $f$ over the segments from $\FF_k$ and $\FF_{k-1}$ respectively). Now we prove the same property for the operators $\Delta_{k,l}$.

\begin{lemma}
	Let $f$ be a function such that $\supp f \subset e_k\in\FF_{k-1}$. Then $\supp \Delta_{k,l} f \subset e_k$, $1\le l\le p_k-1$.
\end{lemma}
\begin{proof}
	Without loss of generality we may assume that  $e_k$ is one of the segments of length $m_{k-1}$ generating the $\sigma$-algebra $\FF_{k-1}$, that is, the segment of the form $[jm_{k-1}^{-1},(j+1)m_{k-1}^{-1})$. Consider any other such segment $e$. We will prove that $\Delta_{k,l} f = 0$ on $e$. As we mentioned above, the function 
	$$
	\Delta_k f=\sum_{l=1}^{p_k-1} \Delta_{k,l} f
	$$
	is equal to $0$ on $e$. Therefore, in order to prove the lemma we only need to show that the functions $\Delta_{k,l} f$, $1\le l\le p_k-1$ are pairwise orthogonal in $L^2(e)$.
	
	In order to do so, it suffices to show the orthogonality in $L^2(e)$ of the functions $w_{n_1}$ and $w_{n_2}$ if $n_1\in[l_1 m_{k-1}, (l_1+1)m_{k-1}-1]$ and $n_2\in[l_2 m_{k-1}, (l_2+1)m_{k-1}-1]$. Such functions $w_{n_1}$ and $w_{n_2}$ are of the following form:
	$$
	w_{n_1}=r_1^{\alpha_1} r_2^{\alpha_2}\ldots r_{k-1}^{\alpha_{k-1}} r_k^{l_1}, \quad w_{n_2}=r_1^{\beta_1} r_2^{\beta_2}\ldots r_{k-1}^{\beta_{k-1}} r_k^{l_2},
	$$
	where $\alpha_1, \ldots, \alpha_{k-1}$ and $\beta_1, \ldots, \beta_{k-1}$ are some nonnegative integers. By the construction, the functions $r_1, \ldots, r_{k-1}$ are constants on $e_k\in\FF_{k-1}$ and the orthogonality of $r_k^{l_1}$ and $r_k^{l_2}$ on $e$ follows from the relation
	$$
	\int_e r_k^{l_1}\bar{r}_k^{l_2}=\int_e r_k^{l_1-l_2}=m_k^{-1}\sum_{s=0}^{p_k-1} e^{\frac{2\pi i(l_1-l_2)s}{p_k}}=0.
	$$
	
\end{proof}

We will need the analogue of Gundy's theorem for Vilenkin systems. Using Lemma 2, we can reformulate the Gundy's theorem in a convenient way. The proof of Gundy's theorem for vector-valued martingales is presented in the paper \cite{KisMart}. %Приведём здесь формулировку (отметим, что эта теорема справедлива только для регулярной фильтрации, и здесь мы ещё раз пользуемся тем, что последовательность $\{p_i\}$ ограничена).

\begin{prop}
	Let  $T$ be a linear operator that maps $\ell^2$-valued functions on $[0,1]$ into scalar ones. Suppose that $T$ is defined on all ``Vilenkin polynomials'', that is, such functions $f$ that $\mathbb{E}_n f = f$ for all sufficiently large values of $n$. Also suppose that the following conditions hold:\\
	1. $\|Tf\|_2\lesssim \|f\|_2$;\\
	2.If a function $f$ is such that $\Delta_0 f=0$ and $\supp\, \Delta_k f\subset e_{k}$ where $e_k\in \FF_{k-1}$, then $\{Tf\neq 0\} \subset \cup_{k\ge 1} e_k$.\\
	In this case $T$ satisfies the weak type $(1,1)$ estimate and therefore is a bounded operator from $L^p(\ell^2)$ to $L^p$ for $1<p\le 2$. 
\end{prop}

Lemma 2 implies that in the second condition in this proposition the operators $\Delta_k$ can be replaced with $\Delta_{k,l}$.

\begin{lemma}
	In Proposition 1 the second condition can be replaced with the following: \\
	2'. If a function $f$ is such that $\Delta_0 f=0$ and $\supp \Delta_{k,l} f \subset e_k \in \FF_{k-1}$ for all $1\le l<p_k-1$, then $\{Tf\neq 0\} \subset \cup_{k\ge 1} e_k$.
\end{lemma}

Indeed, Lemma 2 implies that if $\supp \Delta_k f \subset e_k$, then $\supp\Delta_{k,l} f\subset e_k$. Now we can prove the boundedness on $L^p$ of the operators of the special form which we will use in what follows. Let us denote the interval $[lm_{k-1}, (l+1)m_{k-1}-1]$ by $\delta_{k,l}$.

\begin{corollary}
	Let $\mathcal{A}\subset\Z_+^2$, $h=\{h_{j,k}\}_{(j,k)\in \Z_+^2}\in L^p(\ell^2)$. For any element $(j,k)\in\mathcal{A}$ we fix a set $\Lambda_{\mathcal{A}}\subset [1, p_k-1]$. Suppose that $\{a_{j,k}\}_{(j,k)\in\mathcal{A}}$ is a family of non-negative integers such that $\{a_{j,k}\dotplus \delta_{k,l}\}_{\substack{(j,k)\in\mathcal{A} \\ l\in\Lambda_{\mathcal{A}}}}$ is a family of pairwise disjoint subsets of $\Z_+$. Let us define the operator $G$ in a following way:
	$$
	Gh=\sum_{\substack{(j,k)\in\mathcal{A} \\ l\in\Lambda_{\mathcal{A}}}} w_{a_{j,k}} \Delta_{k,l}h_{j,k}.
	$$
	Then $\|Gh\|_p \lesssim \|h\|_{L^p(\ell^2)}$ for $1<p\le 2$.
\end{corollary}

\begin{proof}
	Since the sets $a_{j,k}\dotplus \delta_{k,l}$ are disjoint, the functions $w_{a_{j,k}} \Delta_{k,l}(h_{j,k})$ are pairwise orthogonal. Hence, we have:
	$$
	\|Gh\|_2^2 = \sum_{\substack{(j,k)\in\mathcal{A} \\ l\in\Lambda_{\mathcal{A}}}} \|w_{a_{j,k}} \Delta_{k,l}h_{j,k}\|_2^2 = \sum_{\substack{(j,k)\in\mathcal{A} \\ l\in\Lambda_{\mathcal{A}}}} \| \Delta_{k,l}h_{j,k}\|_2^2 \le \|h\|_2^2.
	$$
	Besides that, it is obvious that the condition 2' from Lemma 3 holds for $G$ and therefore $G$ is a weak type $(1,1)$ operator and maps $L^p(\ell^2)$ into $L^p$ for $1<p\le 2$.
\end{proof}

Now we are ready to pass to the proof of Theorem 1.

\section{Proof of Theorem 1}

We have the non-intersecting intervals $I_s=[a_s, b_s)\subset \N_0$. The proof of Theorem 1 will consist of two parts: the decomposition of each interval into smaller parts (which is the generalization of the construction from the paper \cite{Osip}) and the application of this decomposition to the estimate of $L^p$-norm of the function $\sum_s f_s$.

\subsection{The decomposition of intervals}

Let us omit the index $s$ in our notation and describe the decomposition of an interval $I=[a,b)$. Suppose that $b$ has the following $(m)$-radix representation:
$$
b=\beta_{k+1}m_k +\beta_{k}m_{k-1}+\ldots+\beta_1.
$$
At first, we decompose the interval $[0,b)$:
$$
[0,b)=\bigcup_{j=1}^{k+1} J_j, \quad  \text{where} \ J_j=[\beta_{k+1}m_k +\ldots+\beta_{j+1}m_j, \beta_{k+1}m_k +\ldots+\beta_j m_{j-1}-1].
$$
In particular, $J_{k+1}$ is the interval $[0, \beta_{k+1}m_{k}-1]$.
If $\beta_j=0$ for some $j$, then we imply that $J_j$ is empty. Note that $J_j$ consists of the numbers which have the following $(m)$-radix representation:
\begin{equation}
	\label{rep1}
	J_j \sim \begin{pmatrix}
		m_k& \ldots & m_j & m_{j-1}& m_{j-2}& \ldots & m_0\\
		\beta_{k+1} & \ldots & \beta_{j+1} & [0, \beta_j-1]& * & \ldots & *
	\end{pmatrix}.
\end{equation}
This notation means that the numbers in $J_j$ can be represented as $$\beta_{k+1}m_k + \ldots +\beta_{j+1}m_j + \gamma_j m_{j-1}+ \varepsilon_{j-1}m_{j-2}+\ldots + \varepsilon_1,$$ where $\gamma_j\in [0, \beta_j-1]$ and $\varepsilon_i$ is any number in the interval $[0, p_i-1]$, $1\le i\le j-1$. In what follows, we will use such notation without extra clarification.

The number $a$ lies in one of the intervals $J_j$; let $a\in J_t$, $1\le t\le k+1$. In this case, $a$ has the following representation in $(m)$-radix numerical system:
$$
a=\beta_{k+1}m_k+\ldots+\beta_{t+1}m_t + \alpha_t m_{t-1} +\alpha_{t-1}m_{t-2}+\ldots+\alpha_1,
$$
where $\alpha_t < \beta_t$. For the sake of convenience, we set  $\alpha_{k+1}=\beta_{k+1} \ldots, \alpha_{t+1}=\beta_{t+1}$. Now we decompose the interval $[a, \beta_{k+1}m_k + \ldots+\beta_{t+1}m_t+\beta_t m_{t-1}-1]=[a, +\infty)\cap J_t$ in a following way:
$$
[a, +\infty)\cap J_t=\{a\}\cup\bigcup_{j=1}^t \widetilde{J}_j,
$$
where for $1\le j\le t-1$ we have
$$
\widetilde{J}_j = [\alpha_{k+1}m_k+\ldots+\alpha_{j+1}m_j + (\alpha_j+1)m_{j-1}, \alpha_{k+1}m_k+\ldots+\alpha_{j+1}m_j+p_jm_{j-1}-1],
$$
and the interval $\widetilde{J}_t$ is the following:
$$
\widetilde{J}_t=[\alpha_{k+1}m_k+\ldots+\alpha_{t+1}m_t + (\alpha_t+1)m_{t-1}, \alpha_{k+1}m_k+\ldots+\alpha_{t+1}m_t+\beta_tm_{t-1}-1].
$$
If $\alpha_j = p_j-1$, $1\le j\le t-1$, then we imply that the interval $\widetilde{J}_j$ is empty. Similarly, if $\alpha_t = \beta_t-1$, then $\widetilde{J}_t=\emptyset$. As before, we note that the interval $\widetilde{J}_j$  consists of the numbers which have the following form (for $1\le j\le t-1$):
\begin{equation}
	\label{rep2}
	\widetilde{J}_j\sim \begin{pmatrix}
		m_k& \ldots & m_j & m_{j-1}& m_{j-2}& \ldots & m_0\\
		\alpha_{k+1} & \ldots & \alpha_{j+1} & [\alpha_j+1, p_j-1]& * & \ldots & *
	\end{pmatrix}.
\end{equation}
The interval $\widetilde{J}_t$ can be written in such form as
\begin{equation}
	\label{rep3}
	\widetilde{J}_t \sim \begin{pmatrix}
		m_k& \ldots & m_t & m_{t-1}& m_{t-2}& \ldots & m_0\\
		\alpha_{k+1} & \ldots & \alpha_{t+1} & [\alpha_t+1, \beta_t-1]& * & \ldots & *
	\end{pmatrix}.
\end{equation}

So, we constructed the decomposition of an interval $I=[a,b)$:
$$
I=\{a\}\cup\bigcup_{j=1}^t \widetilde{J}_j\cup \bigcup_{j=1}^{t-1}J_j.
$$

\subsection{Completion of the proof}

Now we return the index $s$ to our notation. We appy the decomposition from the previous subsection to each interval  $I_s$:
$$
I_s=\{a_s\}\cup\bigcup_{j=1}^{t_s} \widetilde{J}_{j,s}\cup \bigcup_{j=1}^{t_s-1}J_{j,s}.
$$
We also set $\{a_s\}=:\widetilde{J}_{0,s}$. Now we decompose each function $f_s$ in a following way:
$$
f_s=\sum_{j=0}^{t_s} \widetilde{f}_{j,s}+\sum_{j=1}^{t_s-1}f_{j,s},
$$
where the functions $\widetilde{f}_{j,s}$ and $f_{j,s}$ are defined as
$$
\widetilde{f}_{j,s}=\sum_{n\in \widetilde{J}_{j,s}} (f_s, w_n) w_n,\ 0\le j\le t_s; \quad f_{j,s}=\sum_{n\in J_{j,s}} (f_s, w_n)w_n, \ 1\le j\le t_s-1.
$$

Consider the following functions:
$$
\widetilde{g}_{j,s}=w_{a_s}^{-1}\widetilde{f}_{j,s}, \ 0\le j\le t_s; \quad g_{j,s}=w_{b_s}^{-1}f_{j,s}, \ 1\le j\le t_s-1.
$$
The function $f_s$ can be written in the following way:
\begin{equation}
	\label{repf}
	f_s=w_{a_s}\sum_{j=0}^{t_s}\widetilde{g}_{j,s}+w_{b_s}\sum_{j=1}^{t_s-1} g_{j,s}.
\end{equation}
Note that the non-zero Vilenkin coefficients of functions $\widetilde{g}_{j,s}$ and $g_{j,s}$ are contained in the sets $\widetilde{J}_{j,s}\dotdiv a_s$ and $J_{j,s}\dotdiv b_s$, respectively (these sets turn out to be intervals). From the formulas (\ref{rep1}), (\ref{rep2}), (\ref{rep3}) we see that these sets have the following form (here $1\le j\le t_s-1$):
\begin{align}
	\label{rep2'}
	\widetilde{J}_{j,s}\dotdiv a_s \sim& \begin{pmatrix}
		m_{k_s}& \ldots & m_j & m_{j-1}& m_{j-2}& \ldots & m_0\\
		0 & \ldots & 0 & [1, p_j-1-\alpha_{j,s}]& * & \ldots & *
	\end{pmatrix};\\
	\widetilde{J}_{t_s,s}\dotdiv a_s \sim& \begin{pmatrix}
		m_{k_s}& \ldots & m_{t_s} & m_{t_s-1}& m_{t_s-2}& \ldots & m_0\\
		\label{rep3'}
		0 & \ldots & 0 & [1, \beta_{t_s,s}-1-\alpha_{t_s,s}]& * & \ldots & *
	\end{pmatrix};\\
	\label{rep1'}
	J_{j,s}\dotdiv b_s \sim& \begin{pmatrix}
		m_{k_s}& \ldots & m_j & m_{j-1}& m_{j-2}& \ldots & m_0\\
		0 & \ldots & 0 & [p_j-\beta_{j,s}, p_j-1]& * & \ldots & *
	\end{pmatrix}.
\end{align}
Hence, the equation (\ref{repf}) can be rewritten as
\begin{multline*}
	f_s= w_{a_s}\Big( \Delta_0 \widetilde{g}_{0,s}+\sum_{j=1}^{t_s-1}\sum_{l=1}^{p_j-1-\alpha_{j,s}}\Delta_{j,l}\widetilde{g}_{j,s} +\sum_{l=1}^{\beta_{t_s,s}-1-\alpha_{t_s,s}} \Delta_{t_s, l}\widetilde{g}_{t_s,s} \Big) \\ +w_{b_s}\Big( \sum_{j=1}^{t_s-1} \sum_{l=p_j-\beta_{j,s}}^{p_j-1} \Delta_{j, l}g_{j,s} \Big).
\end{multline*}
We use the Corollary to Lemma 3 and conclude that the following inequality holds:
\begin{equation}
	\label{mainineq1}
	\Big\| \sum_s f_s \Big\|_p\lesssim \Big\| \Big( \sum_s \sum_{j=0}^{t_s} |\widetilde{g}_{j,s}|^2 + \sum_s \sum_{j=1}^{t_s-1} |g_{j,s}|^2 \Big)^{1/2} \Big\|_p.
\end{equation}
The expression in the right hand side of this inequality can be estimated by the quantity
$$
\Big\| \Big( \sum_s \sum_{j=0}^{t_s} |\widetilde{g}_{j,s}|^2 \Big)^{1/2} \Big\|_p+  \Big\|\Big( \sum_s \sum_{j=1}^{t_s-1} |g_{j,s}|^2 \Big)^{1/2} \Big\|_p=: A+B.
$$
Now we estimate $A$ and $B$ separately.

We introduce the notation
$$
\widetilde{g}_s=w_{a_s}^{-1} f_s=\sum_{j=0}^{t_s}\widetilde{g}_{j,s}+w_{a_s}^{-1} \sum_{j=1}^{t_s-1}f_{j,s}.
$$
Using the formulas (\ref{rep2'}), (\ref{rep3'}), (\ref{rep1}) respectively, we obtain the following equations:
\begin{align*}
	&\widetilde{g}_{j,s}=\Delta_j \widetilde{g}_s,\ 0\le j\le t_s-1;\\
	&\widetilde{g}_{t_s,s}=\sum_{l=1}^{\beta_{t_s,s}-1-\alpha_{t_s,s}} \Delta_{t_s, l} \widetilde{g}_s;\\
	&w_{a_s}^{-1} \sum_{j=1}^{t_s-1}f_{j,s}=\Delta_{t_s, \beta_{t_s,s}-\alpha_{t_s,s}}\widetilde{g_s}.
\end{align*}
We conclude that
$$
A\lesssim \Big\| \Big( |\Delta_0 \widetilde{g}_s|^2 +\sum_{j=1}^\infty \sum_{l=1}^{p_j-1} |\Delta_{j,l} \widetilde{g}_s|^2  \Big)^{1/2} \Big\|_p=\|\widetilde{S}(\{\widetilde{g}_s\}_s)\|_p,
$$
where by $\widetilde{S}(\{\widetilde{g}_s\}_s)$ we mean the operator $\widetilde{S}$ applied to $\ell^2$-valued function $\{\widetilde{g}_s\}_s$. Now we simply apply Lemma 1, note that $\|\{g_s\}_s\|_{L^p(\ell^2)}=\|\{f_s\}_s\|_{L^p(\ell^2)}$, and the quantity $A$ is estimated.

The expression $B$ can be estimated similarly. We set
$$
g_s=w_{b_s}^{-1}f_s=w_{b_s}^{-1}\sum_{j=0}^{t_s} \widetilde{f}_{j,s}+\sum_{j=1}^{t_s-1}g_{j,s}.
$$
Now we use the formulas (\ref{rep2}), (\ref{rep3}) and (\ref{rep1'}) in order to conclude that
\begin{align*}
	&w_{b_s}^{-1}\sum_{j=0}^{t_s}\widetilde{f}_{j,s}=\Delta_{t_s}g_s;\\
	&g_{j,s}=\Delta_{j} g_s,\ 1\le j\le t_s-1.
\end{align*}
Hence, we can write the following inequalities: 
$$
B\lesssim \|S(\{g_s\}_s)\|_p\lesssim \|\{g_s\}_s\|_{L^p(\ell^2)}= \|\{f_s\}\|_{L^p(\ell^2)},
$$
and Theorem 1 is proved.

\section{The case $p \le 1$}
In this section we show that the above proof also implies a certain inequality for the values of $p$ not greater than $1$. In order to do so, we need the notion of martingale Hardy spaces $\mathcal{H}^p$. All necessary information about these objects is contained in the book \cite{Weisz}. We note that all statements which we are going to use hold for both scalar-valued and $\ell^2$-valued Hardy spaces.

A function $f$ lies in the Hardy space $\mathcal{H}^p$ if $Sf\in L^p$, $0<p\le 2$. In the case when $p$ is greater than $1$ it is well-known that $\mathcal{H}^p=L^p$, and for $p\le 1$ the Hardy space is not equal to $L^p$. We set $\|f\|_{\mathcal{H}^p}=\|Sf\|_{p}$ (note that for $p<1$ the quantity $\|\cdot\|_p$ is only a quasinorm).

Now we pass to the formulation of the inequality that holds for all values of $p$, $0<p\le 2$.

\begin{theorem}
	Under conditions of Theorem 1 the following inequality holds for the intervals $I_s=[a_s, b_s)$ and any value of $p$,  $0<p\le 2$\emph{:}
	$$
	\Big\| \sum_s f_s \Big\|_p\lesssim \|\{w_{a_s}^{-1}f_s\}_s\|_{\mathcal{H}^p(\ell^2)}+\|\{w_{b_s}^{-1}f_s\}_s\|_{\mathcal{H}^p(\ell^2)}.
	$$
\end{theorem}

\begin{proof}
	First of all, we note that for $p>1$ this theorem is equivalent to Theorem 1. Indeed, we just substitute the norms in Hardy spaces in the right rand side of the inequality with the $L^p(\ell^2)$-norms and infer that $\|\{w_{a_s}^{-1}f_s\}_s\|_{L^p(\ell^2)}=\|\{w_{b_s}^{-1}f_s\}_s\|_{L^p(\ell^2)}=\|\{f_s\}_s\|_{L^p(\ell^2)}$.
	
	Now we briefly show how to change the proof of Theorem 1 in order to obtain the desired inequality for $0<p\le 1$.
	
	At first, we note that any operator $T$ that satisfies the conditions of Lemma 3 acts from the Hardy space $\mathcal{H}^p(\ell^2)$ to $L^p$. This fact easily follows from the atomic decomposition of the functions in Hardy spaces (we again refer the reader to the book \cite{Weisz} which contains all necessary theory for the scalar-valued spaces $\mathcal{H}^p$; as for the spaces, $\mathcal{H}^p(\ell^2)$ all necessary statements hold for them with the same proofs).
	
	This observation allows us to conclude that the inequality (\ref{mainineq1}) holds also for $0<p\le 1$ (since $\Delta_j \widetilde{g}_{j,s}=\widetilde{g}_{j,s}$ and $\Delta_j g_{j,s}=g_{j,s}$, see the formulas (\ref{rep2'}), (\ref{rep3'}), (\ref{rep1'})).
	
	Next, we estimated the expressions $A$ and $B$ by the quantities $\|\widetilde{S}(\{\widetilde{g}_s\}_s)\|_p$ and $\|S(\{g_s\}_s)\|_p$. In order to finish the proof of Theorem 2 we simply note that $\|\widetilde{S}f\|_p\asymp \|Sf\|_p$ for $0<p\le 1$ (it means that the Hardy spaces for Vilenkin systems can also be defined using the quadratic function $\widetilde{S}$ instead of $S$). This fact is also known --- since we have the pointwise inequality $Sf\lesssim\widetilde{S}f$, it is just a reformulation of the fact that the operator $\widetilde{S}$ maps $\mathcal{H}^p$ to $L^p$. And this is true because $\widetilde{S}$ is a sublinear operator which satisfied the conditions of Lemma 3. This is verified in the paper \cite{Sim} (besides that, in this paper it is proved that $\|\widetilde{S}f\|_p\asymp \|Sf\|_p$ for $p=1$). 
\end{proof}

\section{The related open problems}

The results of this paper are obtained only for bounded Vilenkin systems. This is a natural assumption in many questions regarding analysis on Vilenkin groups --- see for example the papers \cite{Sim} and \cite{Wat}. As we already mentioned earlier, the boundedness of Vilenkin system is necessary for the boundedness of the operator $\widetilde{S}$ on $L^p$. Besides that, this assumption implies the regularity of the underlying filtration. Nevertheless, sometimes it is possible to abandon this assumption (see for example the paper \cite{Young}). We do not know whether the theorems from this paper hold for unbounded Vilenkin systems, however our ``combinatorial'' approach uses this assumption heavily.

Besides that, we do not know if the inequality from Theorem 1 holds for $0<p\le 1$. The fact that the inequality (\ref{RdFclass}) is valid in the case $p\le 1$ allows us to hope that the corresponding theorem might be true for Vilenkin systems. However, it seems to be an open question even for the particular case of Walsh functions.

\bigskip

Saint Petersburg Leonard Euler International Mathematical Institute, Fontanka 27, St.  Petersburg 191023, Russia

St. Petersburg Department, Steklov Math. Institute, Fontanka 27, St. Petersburg 191023 Russia

\medskip

celis-anton@yandex.ru
%\begin{thebibliography}{99}

%\bibitem{Osipov} N. N. Osipov, \emph{Littlewood–Paley–Rubio de Francia inequality for the Walsh system}, Алгебра и анализ, \textbf{28}:5 (2016), 236--246

%\bibitem{RdF} Jos\'e L. Rubio de Francia, \emph{A Littlewood--Paley inequality for arbitrary intervals}, Rev. Mat. Iberoamericana, \bf{1} (1985), no. 2 

%\end{thebibliography}

%\begin{bibdiv}
%\begin{biblist}

%	\bib{Osipov}{article}{
%		author={Osipov, N. N.},
%		title={Littlewood--Paley--Rubio de Francia inequality for the Walsh system},
%		journal={Алгебра и анализ},
%		volume={28},
%		date={2016},
%		number={5}*{language={russian}},
%		pages={236--246}
%	}

%	\bib{RdF}{article}{
%		author={Rubio de Francia, Jos\'{e} L.},
%		title={A Littlewood-Paley inequality for arbitrary intervals},
%		journal={Rev. Mat. Iberoamericana},
%		volume={1},
%		date={1985},
%		number={2},
%		pages={1--14},
%	}
%\end{biblist}
%\end{bibdiv}

\end{document}